\documentclass[12pt]{amsart}
\usepackage{amsmath,amssymb,amsfonts,amsthm}
\usepackage{epsfig}
\usepackage{a4wide}
\usepackage{enumerate,color}
\usepackage{subfigure,hyperref}
\newtheorem{theorem}{Theorem}[section]

\newtheorem{corollary}[theorem]{Corollary}
\newtheorem{prob}[theorem]{Problem}

\newtheorem{claim}[theorem]{Claim}

\theoremstyle{definition}

\def\cE{\mathcal{E}}

\def\bN{\mathbb{N}}

\newcommand{\Ex}{\mathbb{E}}

\newcommand{\depth}{\operatorname{depth}}

\def\longequation{$$\vcenter\bgroup\advance\hsize by -9em%
\noindent\ignorespaces\refstepcounter{equation}}%
\makeatletter%
\def\endlongequation{\egroup\eqno(\theequation)$$\global\@ignoretrue}
\makeatother

\let\oldrceil\rceil
\renewcommand{\rceil}{\right\oldrceil}
\let\oldlceil\lceil
\renewcommand{\lceil}{\left\oldlceil}

\begin{document}

\title{Exact distance coloring in trees}

\author[N. Bousquet]{Nicolas Bousquet}
\address{Univ. Grenoble Alpes, CNRS, G-SCOP, Grenoble, France}
\email{nicolas.bousquet@grenoble-inp.fr}

\author[L. Esperet]{Louis Esperet} 
\address{Univ. Grenoble Alpes, CNRS, G-SCOP, Grenoble, France}
\email{louis.esperet@grenoble-inp.fr}

\author[A. Harutyunyan]{Ararat Harutyunyan} 
\address{LAMSADE, University of Paris-Dauphine, Paris, France}
\email{ararat.harutyunyan@dauphine.fr}

\author[R. de Joannis de Verclos]{R\'emi de Joannis de Verclos}
\address{Radboud University Nijmegen, Netherlands}
\email{r.deverclos@math.ru.nl}

\thanks{The authors were partially supported by ANR Projects STINT
  (\textsc{anr-13-bs02-0007}) and GATO (\textsc{anr-16-ce40-0009-01}) and LabEx PERSYVAL-Lab
  (\textsc{anr-11-labx-0025-01}) and LabEx CIMI}

\date{}

\sloppy

\begin{abstract}
For an integer $q\ge 2$ and an even integer $d$, consider the graph
obtained from a large
complete $q$-ary tree by connecting with an edge any two vertices at
distance exactly $d$ in the tree. This graph has clique number $q+1$, and
the purpose of this short note is to prove that its chromatic number is $\Theta\big(\tfrac{d
  \log q}{\log d}\big)$. It was not known that the chromatic number of
this graph grows with $d$. As a simple corollary of our result, we give a
negative answer to a problem of Van den Heuvel and Naserasr, asking
whether there is a constant $C$ such that for any odd integer $d$, any
planar graph can be colored with at most $C$ colors such that any pair of
vertices at distance exactly $d$ have distinct colors.
Finally, we study interval coloring of trees (where vertices at
distance at least $d$ and at most $cd$, for some real $c>1$, must be
assigned distinct colors), giving a sharp upper bound in the case of
bounded degree trees. 
\end{abstract}

\maketitle

\section{Introduction}

Given a metric space $X$ and some real $d>0$, let $\chi(X,d)$ be the
minimum number of colors in a coloring of the elements of $X$ such
that any two elements at distance exactly $d$ in $X$ are assigned
distinct colors. The classical Hadwiger-Nelson problem asks for the value
of $\chi(\mathbb{R}^2,1)$, where $\mathbb{R}^2$ is the Euclidean
plane. It is known that $5\le \chi(\mathbb{R}^2,1)\le 7$~\cite{Gre18} and since the Euclidean
plane $\mathbb{R}^2$ is invariant under homothety,
$\chi(\mathbb{R}^2,1)=\chi(\mathbb{R}^2,d)$ for any real
$d>0$. Let $\mathbb{H}^2$ denote the hyperbolic
plane. Kloeckner~\cite{Klo15} proved that $\chi(\mathbb{H}^2,d)$ is at
most linear in $d$ (the multiplicative constant was recently improved
by Parlier and Petit~\cite{PP17}), and observed that
$\chi(\mathbb{H}^2,d)\ge 4$ for any $d>0$. He raised the question of determining whether
$\chi(\mathbb{H}^2,d)$ grows with $d$ or can be bounded independently
of $d$. As noticed by Kahle (see~\cite{Klo15}), it is not known whether
$\chi(\mathbb{H}^2,d)\ge 5$ for some real $d>0$. Parlier and
Petit~\cite{PP17} recently suggested to study infinite regular trees
as a discrete analog of the hyperbolic plane. Note that any graph $G$
can be considered as a metric space (whose elements are the vertices
of $G$ and whose metric is the graph distance in $G$), and in this
context $\chi(G,d)$ is precisely the minimum number of colors in a
vertex coloring of $G$ such that vertices at distance $d$ apart are
assigned different colors. Note that $\chi(G,d)$ can be equivalently
defined as the chromatic number of the \emph{exact $d$-th power} of
$G$, that is,
the graph with the same vertex-set as $G$ in which two vertices are adjacent if
and only if they are at distance exactly $d$ in $G$. 

Let $T_q$ denote the infinite $q$-regular tree. Parlier and
Petit~\cite{PP17} observed that when $d$ is odd, $\chi(T_q,d) =2$ and
proved that when $d$ is even, 
$q\le \chi(T_q,d) \le (d+1)(q-1)$. A similar upper bound can also be
deduced from the results of Van den Heuvel, Kierstead, and Quiroz~\cite{HOQ16}, while the lower bound is a
direct consequence of the fact that when $d$ is even, the clique number of the exact $d$-th
power of $T_q$ is $q$ (note that it does not depend on $d$). In this
short note, we prove that when $q \geq 3$ is fixed, $$\tfrac{d \log(q-1)}{4\log(d/2)+4\log(q-1)}\le \chi(T_q,d) \le 
(2+o(1))\tfrac{d \log(q-1)}{\log d},$$ 

\smallskip

\noindent where the asymptotic $o(1)$ is in terms of $d$.
A simple consequence of our main result is that for any even integer $d$, the exact $d$-th power of
a complete binary tree of depth $d$ is of order $\Theta(d/\log d)$
(while its clique number is equal to 3).

\medskip

The following problem (attributed to Van den Heuvel and Naserasr) was
raised in~\cite{NO12} (see also~\cite{HOQ16} and~\cite{NO15}).

\begin{prob}[Problem 11.1 in~\cite{NO12}]\label{prob:NO}
Is there a constant $C$ such that for every odd integer $d$ and every
planar graph $G$ we have $\chi(G,d)\le C$?
\end{prob}

We will show that our result on large complete binary
trees easily implies a negative answer to Problem~\ref{prob:NO}. More
precisely, we will prove that the graph $U_3^d$ obtained from a complete
binary tree of depth $d$ by adding an edge between any two vertices
with the same parent gives a negative answer to Problem~\ref{prob:NO}
(in particular, for odd $d$, the chromatic number of the exact $d$-th power of $U_3^d$ grows
as $\Theta(d/\log d)$). We will also prove that the exact $d$-th power
of a specific subgraph $Q_3^d$ of $U_3^d$ grows as $\Omega(\log
d)$. Note that $U_3^d$ and $Q_3^d$ are outerplanar (and thus,
planar) and chordal (see Figure~\ref{fig:Q}). 

\medskip

Kloeckner~\cite{Klo15} proposed the following variant of the original problem: 
For a metric space $X$, an integer $d$ and a real $c>1$, we denote by $\chi(X,[d,cd])$ the
smallest number of colors in a coloring of the elements of $X$ such
that any two elements of $X$ at distance at least $d$ and at most
$cd$ apart have
distinct colors. Considering as above the natural metric space defined
by the 
infinite $q$-regular tree $T_q$, Parlier and
Petit~\cite{PP17} proved that $$q(q-1)^{\lfloor cd/2\rfloor -\lfloor
  d/2\rfloor}\le \chi(T_q,[d,cd])\le (q-1)^{\lfloor
  cd/2+1\rfloor}(\lfloor cd \rfloor+1).$$ 

We will show that $\chi(T_q,[d,cd])\le \tfrac{q}{q-2}(q-1)^{\lfloor
cd/2\rfloor -d/2+1}+cd+1$, which implies that the lower bound of Parlier and
Petit~\cite{PP17} (which directly follows from a clique size argument) is
asymptotically sharp.

\section{Exact distance coloring}

Throughout the paper, we assume that the infinite $q$-regular tree $T_q$ is rooted in
some vertex $r$. This naturally
defines the children and descendants of a vertex and the parent and
ancestors of a vertex distinct from $r$. In particular, given a vertex
$u$, we define the ancestors $u^0,u^1,\ldots$ of $u$ inductively as
follows: $u^0=u$ and for any $i$ such that $u^i$ is not the root,
$u^{i+1}$ is the parent of $u^i$. With this notation, $u^d$ can be
equivalently defined as the ancestor of $u$ at distance $d$ from $u$
(if such a vertex exists). For a given vertex $u$ in $T_q$, the \emph{depth} of
$u$, denoted by $\depth(u)$, is the distance between $u$ and $r$ in $T_q$.
For a
vertex $v$ and an integer $\ell$, we define $L(v,\ell)$ as the set of descendants of $v$ at distance
exactly $\ell$ from $v$ in $T_q$.

\medskip

We first prove an upper bound on $\chi(T_q,d)$.

\begin{theorem}\label{thm:upext}
For any integer $q\ge 3$, any even integer $d$, and any integer
$k\ge 1$ such that $k(q-1)^{k-1}\le d$, we have $\chi(T_q,d)\le (q-1)^{k}+(q-1)^{\lfloor k/2
  \rfloor}+\tfrac{d}{k} +1$. In particular, $\chi(T_q,d)\le
d+q + 1$, and when $q$ is fixed and $d$ tends to infinity,
$\chi(T_q,d)\le (2+o(1))\tfrac{d \log(q-1)}{\log d}$.
\end{theorem}

\begin{proof}
A vertex of $T_q$ distinct from $r$ and whose depth is a multiple of $k$ is said to
be a \emph{special vertex}. Let $v$ be a special vertex. Every special
vertex
$u$ distinct from $v$
such that $u^k=v^k$ is called a \emph{cousin} of $v$. Note that $v$ has at most $q(q-1)^{k-1}-1$
cousins (at most $(q-1)^{k}-1$ if $v^k\ne r$). A special vertex $u$ is said to be a \emph{relative} of $v$ if $u$ is
either a cousin of $v$, or $u$ has the
property that $u$ and $v^k$ have the same depth and are at distance at
most $k$ apart in $T_q$. 
Two vertices $a,b$ at distance at most $k$ apart and at the same depth must satisfy
  $a^{\lfloor k/2 \rfloor}=b^{\lfloor k/2 \rfloor}$, and so the number of vertices $u$ such that $u$ and $v^k$ have the same depth and are at distance at
most $k$ apart in $T_q$ is $(q-1)^{\lfloor k/2\rfloor}$.
It follows that if $v^k=r$, then $v$ has at most
$q(q-1)^{k-1}-1$
relatives and otherwise $v$ has at most $(q-1)^{k}+(q-1)^{\lfloor
  k/2\rfloor}-1$ relatives.

The first step is to define a coloring $C$ of the special vertices of
$T_q$. This will be used later to define the desired coloring of
$T_q$, i.e.\ a coloring such that vertices of $T_q$ at distance $d$ apart are assigned
distinct colors (in this second coloring, the special vertices will
not retain their color from $C$).

\smallskip

We greedily assign a color $C(v)$ to
each special vertex $v$ of $T_q$ as follows: we consider the vertices of $T_q$ in a
breadth-first search starting at $r$, and for each special vertex $v$
we encounter, we assign to $v$ a color distinct from the colors already
assigned to its relatives, and from the set of ancestors $v^{ik}$ of $v$,
where $2\le i \le \tfrac{d}{k}+1$ (there are at most $\tfrac{d}{k}$
such vertices).
Note that if $v^k=r$, the number of colors forbidden
for $v$ is at most $q(q-1)^{k-1}-1$ and if
$v^k\ne r$ the number of colors forbidden for $v$ is at most $(q-1)^{k}+(q-1)^{\lfloor k/2\rfloor}+\tfrac{d}{k}-1$. 
Since $k(q-1)^{k-1}\le d$, in both cases
$v$ has at most $(q-1)^{k}+(q-1)^{\lfloor k/2\rfloor}+\tfrac{d}{k}-1$
forbidden colors, therefore we can obtain the coloring $C$ by using
at most $(q-1)^{k}+(q-1)^{\lfloor k/2\rfloor}+\tfrac{d}{k}$ colors.

\smallskip

For any special vertex $v$, the set of descendants of $v$ at distance at least $d/2-k$ and at most
$d/2-1$ from $v$ is denoted by $K(v,k)$.
We now define
the desired coloring of $T_q$ as follows: for each special vertex $v$,
all the vertices of $K(v,k)$ are
assigned the color $C(v)$. Finally, all the vertices at distance at
most $d/2-1$ from $r$ are colored with a single new color (note that any two
vertices in this set lie at distance less than $d$ apart). The resulting
vertex-coloring of $T_q$ is called $c$. Note that $c$ uses at most
$(q-1)^{k}+(q-1)^{\lfloor k/2\rfloor}+\tfrac{d}{k}+1$ colors, and indeed
every vertex of $T_q$ gets exactly one color.

We now prove that vertices at distance
$d$ apart in $T_q$ are assigned distinct colors in $c$.
Assume for the sake of contradiction that two vertices $x$ and $y$ at distance $d$ apart were
assigned the same color. Then the depth of both $x$ and $y$ is at
least $d/2$. We can assume by symmetry  that
the difference $t$ between the depth of $x$ and the depth of $y$ is such that $0\le t \le d$
since otherwise they would be at distance more than $d$.
Let $u$ be the unique (special) vertex of $T_q$ such that $x \in
K(u,k)$ and $v$ be the unique (special) vertex such that $y\in
K(v,k)$. By the definition of our coloring $c$, we have $C(u)=C(v)$. Note that $u$ and $v$ are distinct;
indeed, otherwise $x$ and $y$
would not be at distance $d$ in $T_q$. Assume first that
$u$ and $v$ have the same depth. Then since $u$ and $x$ (resp. $v$ and
$y$) are distance at least $d/2-k$ apart, $u$ and $v$ are
cousins (and thus, relatives), which contradicts the definition of the vertex-coloring
$C$. We may, therefore, assume that the depths of $u$ and $v$ are distinct. Moreover, since $u$ and $v$ are special
vertices, we may assume that their depths differ 
by at least $k$. In particular, $u$ lies deeper than $v$ in $T_q$. 

First assume
that the depths of $u$ and $v$ differ by at least $2k$. Then $v$ is not an ancestor of $u$ in $T_q$.
Indeed, for otherwise
we would have $v=u^{ik}$ for some integer $2\le i \le \tfrac{d}{k}+1$,
which would contradict the definition of $C$. This implies that the
distance between $x$ and $y$ is at least $d/2-k+d/2-k+2k+2=d+2$, which
is a contradiction. Therefore, we can assume that the depths of $u$ and $v$ differ by
precisely $k$. Since $v$ is not a relative of $u$, we have that $v\ne u^k$
and the distance between $u^k$ and $v$ is more than $k$. Moreover, since $u$ and $x$ (resp. $v$ and
$y$) are at distance at least $d/2-k$ apart, this implies that the distance between $x$
and $y$ is more than $d/2 - k + k + k + d/2-k = d$, a contradiction.
%Hence, $v$ is a cousin and then a relative of $u$ in $T_q$, which contradicts the definition of the vertex-coloring $C$.

Thus, $c$ is a proper coloring.

\medskip

By taking $k=1$ we obtain a coloring $c$ using 
at most $(q-1)^{1}+(q-1)^{\lfloor 1/2
  \rfloor}+\tfrac{d}{1} +1=q+d + 1$ colors, and by taking
$k=\lfloor \tfrac{\log d-\log \log d+\log \log (q-1)}{\log (q-1)}\rfloor$,
we obtain a coloring $c$ using at most $$\tfrac{d \log(q-1)}{\log
  d}+\sqrt{\tfrac{d \log(q-1)}{\log d}}+\tfrac{d \log(q-1)}{\log
  d-\log \log d+\log \log (q-1)-\log (q-1)}+1=(2+o(1)) \tfrac{d \log(q-1)}{\log
  d}$$ colors.
\end{proof}

For $k=1$, the proof above can be optimized to show that
$\chi(T_q,d)\le q+\tfrac{d}{2}$ (by simply noting that vertices at even
depth and vertices at odd depth can be colored independently). Since
we are mostly interested in the asymptotic behaviour of $\chi(T_q,d)$
(which is of order $O\big(\tfrac{d}{\log d}\big)$),
we omit the details.

\medskip

We now prove a simple lower bound on $\chi(T_q,d)$.  Let $T_q^d$ be the rooted complete $(q-1)$-ary tree of
depth $d$, with root $r$. Note that each node has $q-1$ children, so this
graph is a subtree of $T_q$.

\begin{theorem}\label{thm:loext1}
For any integer $q\ge 3$ and any even $d$, $\chi(T_q^d,d)\ge \log_2(\tfrac{d}4+q-1)$.
\end{theorem}

\begin{proof}
Consider any coloring of $T_q^d$ with colors $1,2,\ldots, C$, such that
vertices at distance precisely $d$ apart have distinct colors. For any
vertex $v$ at depth at most $\tfrac{d}2+1$ in $T_q^d$, the set of colors appearing in
$L(v,\tfrac{d}2-1)$ is denoted by $S_v$. Observe that if $v$ and $w$
have the same parent, then $S_v$ and $S_w$ are disjoint
since for any $ x\in L(v,\tfrac{d}2-1)$ and $y\in L(w,\tfrac{d}2-1)$,
$x$ and $y$ are at distance $d$.

\smallskip

Fix some vertex $u$ at depth at most $\tfrac{d}2$ in $T_q^d$ and some child $v$ of $u$. We claim that:

\begin{claim}\label{cl:colorchild}
For any integer $1\le k \le \tfrac{\depth(u)}2$, there is a color of
$S_{u^{2k-1}}$ that does not appear in $S_v$.
\end{claim}

To see that Claim~\ref{cl:colorchild} holds, observe that in the subtree of $T_q^d$
rooted in $u^k$, there is a vertex of $L(u^{2k-1},\tfrac{d}2-1)$ at distance $d$ from all the
elements of $L(v,\tfrac{d}2-1)$. The color of such a vertex does not
appear in $S_v$, therefore Claim~\ref{cl:colorchild} holds.

\medskip

In particular, Claim~\ref{cl:colorchild} implies that all the sets
$\{S_{u^{2k-1}}\, | \, 1\le k \le d/4\}$ and $\{S_w\,|\, w \mbox{ is a
child of }u\}$ are
pairwise distinct. Since there are $\tfrac{d}4+q-1$ such sets, we have
$\tfrac{d}4+q-1\le 2^C$ and therefore $C\ge \log_2(\tfrac{d}4+q-1)$, as desired.
\end{proof}

It was observed by St\'ephan Thomass\'e that the proof of
Theorem~\ref{thm:loext1} only uses a small fraction of the graph
$T_q^d$. Consider for simplicity the case $q=3$, and define $P_3^d$ as
the graph obtained from a path $P=v_0,v_1,\ldots,v_d$ on $d$ edges, by
adding, for each $1\le i\le d$, a path on $i$ edges ending at $v_i$
(see Figure~\ref{fig:P}). This graph is an induced subgraph of $T_q^d$ and the
proof of Theorem~\ref{thm:loext1} directly shows the following\footnote{St\'ephan
Thomass\'e noticed that this can also be deduced from the fact that
the vertices at depth at least $\tfrac{d}2$ and at most $d$ in
the exact $d$-th power of $P_3^d$ induce a \emph{shift graph}.}.

\begin{corollary}\label{cor:loext}
For any even integer $d$, $\chi(P_3^d,d)\ge \log_2(d+8)-2$. 
\end{corollary}

\begin{figure}[htbp]
\begin{center}
\includegraphics[scale=1]{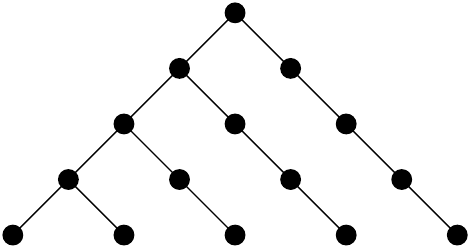}
\caption{The graph $P_3^4$.\label{fig:P}}
\end{center}
\end{figure}

The proof
of Theorem~\ref{thm:loext1} can be refined to prove the following 
better estimate for $T_q^d$, showing that the upper bound of
Theorem~\ref{thm:upext} is (asymptotically) tight within a constant multiplicative
factor of 8.

\begin{theorem}\label{thm:loext2}
For any integer $q\ge 3$ and
every even integer $d\ge 2$, $\chi(T_q^d,d)\ge \tfrac{d \log(q-1)}{4\log(d/2) + 4 \log(q-1)}$.
\end{theorem}

\begin{proof}
Consider any coloring of $T_q^d$ with colors $1,2,\ldots, C$, such that
vertices at distance precisely $d$ apart have distinct colors. 
We perform a random walk $v_0,v_1,\ldots,v_d$ in $T_q^d$ as follows: we
start with $v_0=r$, and for each $i\ge 1$, we choose a child of $v_i$
uniformly at random and set it as $v_{i+1}$. Note that the depth of
each vertex $v_i$ is precisely $i$.

\smallskip

From now on we fix a color $c\in \{1,\ldots,C\}$. For any vertex $v$
of $T_q^d$, the set of vertices contained in the subtree of $T_q^d$ rooted in $v$
is denoted by $V_v$, and we set $A_v=\{\mbox{depth}(u)\,|\, u \in V_v
\mbox{ and } u\mbox{ has color } c\}$. When $v=v_i$, for some integer
$0\le i \le d$, we write $A_i$ instead of\nolinebreak\ $A_{v_i}$.

\begin{claim}\label{cl:obs1}
Assume that for some even integers $i$ and $j $ with $2 \le i<j \le d$, and for some vertex $v$ at depth
$\tfrac{i+j-d}{2}$, the set $A_v$ contains both $i$
and $j$. Then $v$ has precisely one child $u$ such that $A_u$ contains
$i$ and $j$, and moreover all the children $w$ of $v$ distinct from
$u$ are such that $A_w$ contains neither $i$ nor $j$.
\end{claim}

To see that Claim~\ref{cl:obs1} holds, simply note that
$\tfrac{i+j-d}{2}<i<j$ and if two
vertices $u_1,u_2$ colored $c$ are respectively at depths $i$ and $j$, and their
common ancestor is $v$, then they are at distance $d$ in $T_q^d$ (which
contradicts the fact that they were assigned the same color). 
Indeed, the distance of $u_1$ to $v$ is $i-\tfrac{i+j-d}{2}$ and the distance
of $u_2$ to $v$ is $j-\tfrac{i+j-d}{2}$. This proves the claim.

\medskip

We now define a family of graphs $(G_k)_{0\le k \le d/2}$ as
follows. For any $0\le k \le \tfrac{d}2$, the vertex-set $V(G_k)$ of $G_k$ is
the set $A_k\cap 2\bN \cap (d/2,d]$, and
two (distinct) even integers $i,j\in A_k$ are adjacent in $G_k$ if and only if
$\tfrac{i+j-d}2<k$. For each
$0\le k\le \tfrac{d}2$ we define the \emph{energy} $\cE_k$ of $G_k$ as
follows: $\cE_k=\sum_{i\in V(G_k)} (q-1)^{\text{deg} (i)}$, where
$\text{deg} (i)$ denotes the degree of the vertex $i$ in $G_k$.

\medskip

Note that each graph $G_k$ depends on the (random) choice of
$v_1,v_2,\ldots,v_k$. 

\begin{claim}\label{cl:obs2}
For any $0\le k\le
\tfrac{d}2-1$, $\Ex (\cE_{k+1} )\le  \Ex (\cE_{k} )$.
\end{claim}

Assume that $v_1,v_2,\ldots,v_k$ (and therefore also $G_k$) are
fixed. Observe that $G_{k+1}$ is obtained from $G_k$ by
possibly removing some vertices and adding some edges. Thus, $\cE_{k+1}$
can be larger than $\cE_{k}$ only if $G_{k+1}$ contains edges
that are not in $G_k$. Therefore, it suffices to consider the contributions 
of those pairs of nonadjacent vertices in $G_k$ which could become
adjacent in $G_{k+1}$ (since these correspond to pairs $i,j$ with $k=\tfrac{i+j-d}2$,
these pairs are pairwise disjoint), and prove that these contributions are, in expectation, equal to 0.  
Fix a pair of even integers $i<j$ in $V(G_k)$ with
$k=\tfrac{i+j-d}2$ (and note that $i$ and $j$ are not adjacent in
$G_k$). By Claim~\ref{cl:obs1}, either $v_{k+1}$ is such that $A_{k+1}$
contains $i$ and $j$ (this event occurs with probability
$\tfrac{1}{q-1}$), or $A_{k+1}$ contains neither $i$ nor $j$ (with
probability $1-\tfrac1{q-1}$). As a consequence, for any $i<j$ in $V(G_k)$ with
$k=\tfrac{i+j-d}2$, with probability
$\tfrac{1}{q-1}$ we add the edge $ij$ in $G_{k+1}$ and with
probability  $1-\tfrac1{q-1}$ we remove vertices $i$ and $j$ from
$G_{k+1}$. This implies that for any $i,j \in V(G_k)$, $i<j$, with
$k=\tfrac{i+j-d}2$, with probability
$\tfrac{1}{q-1}$ we have contribution at most $(q-1)^{\text{deg} (i)+1}+(q-1)^{\text{deg}
  (j)+1}-(q-1)^{\text{deg} (i)}-(q-1)^{\text{deg}
  (j)}=(q-2)((q-1)^{\text{deg} (i)}+(q-1)^{\text{deg} (j)})$ to
$\cE_{k+1}$ (where $\text{deg}$ refers to the degree in $G_k$) and with
probability  $1-\tfrac1{q-1}$ we have a contribution of at most $-(q-1)^{\text{deg}
  (i)}-(q-1)^{\text{deg} (j)}$ to $\cE_{k+1}$.
  Thus, the expected contribution of such a pair $i,j$ is at most
  $\tfrac1{q-1} (q-2)((q-1)^{\text{deg} (i)}+(q-1)^{\text{deg}
  (j)})-\tfrac{q-2}{q-1}((q-1)^{\text{deg} (i)}+(q-1)^{\text{deg}
  (j)}) = 0. $
  
Summing over all such pairs $i,j$, we obtain $\Ex (\cE_{k+1})\le \Ex (\cE_{k} ).$
This proves Claim~\ref{cl:obs2}.

\medskip

Since $2 \le i<j \le d$, we have $\tfrac{i+j-d}{2}\le \tfrac{d}2-1$, and in
particular it follows that $G_{d/2}$ is a (possibly
empty) complete graph, whose number of vertices is denoted by
$\omega\ge 0$. Note that the energy $\cE$ of a complete graph on
$\omega$ vertices is equal to $\omega (q-1)^{\omega-1}$, while the
energy $\cE_0$ of $G_0$ is equal to $|A_0\cap 2\bN \cap (d/2, d]|\le \tfrac{d}{4}$.
For a vertex $u\in L(r,\tfrac{d}2)$, let $\omega_u=|A_u\cap 2\bN \cap (d/2, d]|$
(this is the number of distinct even depths at which a vertex colored
$c$ appears in the subtree of height $\tfrac{d}2$ rooted in $u$). It follows
from Claim~\ref{cl:obs2} that the average of $\omega_u (q-1)^{\omega_u-1}$, over all
vertices $u\in L(r,\tfrac{d}2)$, is at most $\tfrac{d}4$. Let $a$ be
the average of $\omega_u$, over all vertices  $u\in
L(r,\tfrac{d}2)$. By 
Jensen's inequality and the convexity of the function $x\mapsto
x(q-1)^{x-1}$ for $x\ge 0$, we have that $a(q-1)^{a-1}\le
\tfrac{d}4$, and thus $a\le \tfrac{\log(d/2)}{\log(q-1)} +
1$.

Note that $a$ depends on the color $c$ under
consideration (to make this more explicit, let us now write $a_c$ instead of $a$). Since there are
$\tfrac{d}4$ even depths between depth $\tfrac{d}2$ and depth $d$,
there is a color $c\in \{1,\ldots,C\}$ such that $a_c\cdot C\ge
\tfrac{d}4$ and thus, $C\ge \tfrac{d}{4a_c}\ge \tfrac{d \log(q-1)}{4\log(d/2) + 4 \log(q-1)}$, as desired.
\end{proof}

We now explain how the results proved above give a negative answer to Problem~\ref{prob:NO}.
Let $U_3^d$ (resp. $Q_3^d$) be obtained from $T_3^d$ (resp. $P_3^d$) by
adding an edge $uv$ for any pair of vertices $u,v$ having the same
parent. Note that for any $d$, $U_3^d$ and $Q_3^d$ are outerplanar (and thus,
planar) and chordal, and $Q_3^d$ has pathwidth 2 ($U_3^3$ and $Q_3^5$ are depicted
in Figure~\ref{fig:Q}) and the original copies of $T_3^d$ and $P_3^d$
are spanning
trees of $U_3^d$ and $Q_3^d$, respectively. In the remainder of this section, whenever we write
$T_3^d$, we mean \emph{the original copy of $T_3^d$ in $U_3^d$}.

\begin{figure}[htbp]
\begin{center}
\includegraphics[scale=1]{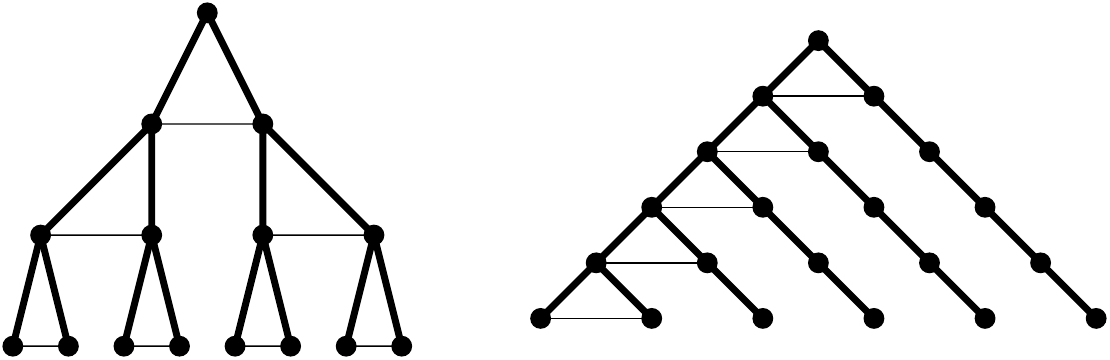}
\caption{The graphs $U_3^3$ (left) and $Q_3^5$ (right). The bold edges represent the
  original copies of $T_3^3$ and $P_3^5$, respectively. \label{fig:Q}}
\end{center}
\end{figure}

Observe that for any two vertices $u$ and $v$ distinct from the root of
$T_3^d$, $u$ and $v$ are at distance $d$ in $T_3^d$ if and only if
they are at distance $d-1$ in $U_3^d$ (since the depth of $T_3^d$ is
$d$, the fact that $u$ and $v$ differ from the root and are at
distance $d$ apart implies that none of the two vertices is an ancestor of the
other). The same property holds for $Q_3^d$ and $P_3^d$. As a consequence, for any odd integer $d$, $\chi(U_3^{d+1},d)$
and $\chi(T_3^{d+1},d+1)$ differ by at most one, and $\chi(Q_3^{d+1},d)$
and $\chi(P_3^{d+1},d+1)$ also differ by at most one. Using this observation, we immediately obtain the following
corollary of Theorem~\ref{thm:loext2} and Corollary~\ref{cor:loext}, which gives a negative answer
to Problem~\ref{prob:NO}.

\begin{corollary}\label{cor:vdhn}
For any odd integer $d$, $$\chi(U_3^{d+1},d)\ge \tfrac{(d+1)
  \log(2)}{4\log((d+1)/2)+4\log(2)}-1 \mbox{ and }\chi(Q_3^{d+1},d)\ge \log_2(d+8)-3. $$
\end{corollary}

The graphs $U_3^{d+1}$ and its exact $d$-th power have $n=2^{d+2}$
vertices, and thus the chromatic number of the exact $d$-th power of
$U_3^{d+1}$ grows as $\Omega\big(\tfrac{\log n}{\log \log
  n}\big)$. The graphs $Q_3^{d+1}$ and its exact $d$-th power have
$n={d+2 \choose 2}$ vertices, and thus the chromatic number of the exact $d$-th power of
$Q_3^{d+1}$ grows as $\Omega({\log n})$. It is not difficult (using
Theorem~\ref{thm:upext} for $U_3^{d+1}$) to show that these bounds are
asymptotically tight.

\smallskip

It was recently proved by Quiroz~\cite{Qui17} that if $G$ is a chordal graph of
clique number at most $t\ge 2$, and $d$ is an odd number, then
$\chi(G,d)\le {t \choose 2}(d+1)$. By Corollary~\ref{cor:vdhn}, the
graph $U_3^d$ shows that this is asymptotically best possible (as $d$
tends to infinity), up to a
$\log d$ factor.

\section{Interval coloring}

For an integer $d$ and a real $c>1$, recall that $\chi(T_q,[d,cd])$
denotes the
smallest number of colors in a coloring of the vertices of $T_q$ such
that any two vertices of $T_q$ at distance at least $d$ and at most
$cd$ apart have
distinct colors. Parlier and
Petit~\cite{PP17} proved that $$q(q-1)^{\lfloor cd/2\rfloor -\lfloor
  d/2\rfloor}\le \chi(T_q,[d,cd])\le (q-1)^{\lfloor
  cd/2+1\rfloor}(\lfloor cd \rfloor+1).$$ In this final section, we
prove that their lower bound (which is proved by finding a set of
vertices of this cardinality that are pairwise at distance at least $d$ and at most
$cd$ apart in $T_q$) is asymptotically tight.

\begin{theorem}\label{thm:upint}
For any integers $q\ge 3$ and $d$ and any real $c>1$, $\chi(T_q,[d,cd])\le \tfrac{q}{q-2}(q-1)^{\lfloor
cd/2\rfloor -d/2+1}+cd+1$.
\end{theorem}

\begin{proof}
The proof is similar to the proof of Theorem~\ref{thm:upext}. 
We consider any ordering $e_1, e_2, \ldots$ of the edges of $T_q$ obtained from a
breadth-first search starting at $r$. Then, for any $i=1,2,\ldots$ in
order, we assign a color $c(e_i)$ to the edge $e_i$ as follows. Let
$e_i=uv$, with $u$ being the parent of $v$, and let $\ell=\lfloor
cd/2\rfloor -d/2$. We assign to $uv$ a color $c(uv)$ distinct from the
colors of all the edges $xy$ (with $x$ being the parent of $y$) such
that $x$ is at distance at most $\ell$ from $u^k$ (where $k$ is the
minimum of $\ell$ and the depth of $u$), or $x$ is an
ancestor of $u$ at distance at most $cd$ from $u$ (and $y$ lies on the
path from $u$ to $x$). There are at most $cd+\sum_{j=0}^\ell q(q-1)^j\le\tfrac{q}{q-2}(q-1)^{\ell+1}+d-1$ such edges, so we can color all the edges
following this procedure by using a total of at most
$\tfrac{q}{q-2}(q-1)^{\ell+1}+cd$ colors.

As in the proof of Theorem~\ref{thm:upext}, we now define our coloring of the
vertices of $T_q$ as follows: first color all the vertices at distance
at most $\tfrac{d}2-1$ from $r$ with a new color that does not appear
on any edge of $T_q$, then for each vertex $v$ with parent $u$, we
color all the vertices of $L(v,\tfrac{d}2-1)$ with color $c(uv)$. In
this vertex-coloring, at most $\tfrac{q}{q-2}(q-1)^{\ell+1}+cd+1$ colors
are used.

Assume that two vertices $s$ and $t$, at distance at least $d$ and at
most $cd$ apart, were assigned the same color. This implies that $c(s^{d/2-1}s^{d/2})=c(t^{d/2-1}t^{d/2})$. Assume without loss of
generality that the depth of $s$ is at least the depth of $t$, and
consider first the case where $t^{d/2-1}$ is an ancestor of $s$. Then
$t^{d/2}$ is an ancestor of $s^{d/2}$ at distance at most $cd$ from
$s^{d/2}$ (and $t^{d/2-1}$ lies on the path from $s^{d/2}$ to $t^{d/2}$), which contradicts the definition of our
edge-coloring $c$. Thus, we can assume that $t^{d/2-1}$ is not an
ancestor of $s$. This implies that $t^{d/2-1}t^{d/2}$ lies on the path
between $s$ and $t$, and therefore $t^{d/2}$ is at distance at most
$\ell=\lfloor
cd/2\rfloor -d/2$ from the ancestor of $s^{d/2}$ at distance $\ell$ from
$s^{d/2}$ (or simply from $r$, if the depth of $s^{d/2}$ is at most $\ell$). Again, this contradicts the definition of our coloring
$c$. We obtained a coloring of the vertices of $T_q$ with at most
$\tfrac{q}{q-2}(q-1)^{\ell+1}+cd+1$ colors in which each pair of vertices
at distance at least $d$ and at most $cd$ apart have distinct colors,
as desired. 
\end{proof}

\section*{Acknowledgement}

We are very grateful to Lucas Pastor, St\'ephan Thomass\'e, and an
anonymous reviewer for their excellent observations and comments.

\end{document}